\documentclass[12pt]{article}
\usepackage{amsmath,amsthm,amssymb,amsfonts}
\usepackage{mathabx,epsfig}

\usepackage[margin=1in]{geometry}
\usepackage{amsmath,amscd}
\usepackage{mathtools}

\usepackage[all]{xy}

\usepackage{thmtools}
\usepackage{thm-restate}
\usepackage{hyperref}
\usepackage{cleveref}

\newtheorem{theorem}{Theorem}

\newtheorem{corollary}{Corollary}
\newtheorem{lemma}{Lemma}
\newtheorem{proposition}{Proposition}
\usepackage[small,compact]{titlesec}
\usepackage{enumerate}
\usepackage{times}
\makeatletter

\newcommand{\Rmnum}[1]{\expandafter\@slowromancap\romannumeral #1@}
\makeatother

 \newcommand{\ord}{\operatorname{ord}}

 \newcommand{\Aut}{\operatorname{Aut}}

\newcommand{\Gal}{\operatorname{Gal}}

\newcommand{\disc}{\operatorname{disc}}

\usepackage{xcolor}
\usepackage{titlesec}
\titleformat{\section}[block]{\color{black}\Large\bfseries\filcenter}{}{1em}{}
\titleformat{\subsection}[hang]{\bfseries}{}{1em}{}
\theoremstyle{remark}
\newtheorem{definition}{Definition}
\newtheorem{remark}{Remark}
\begin{document}
\title{Galois uniformity in quadratic dynamics over $k(t)$}
\author{Wade Hindes\\
Department of Mathematics, Brown University\\
Providence, RI 02912\\
Contact Information: (401)-383-3739, whindes@math.brown.edu}
\date{\today}
\maketitle
\renewcommand{\thefootnote}{}
\begin{abstract}  We prove that the arboreal Galois representation attached to a large class of quadratic polynomials defined over a field of rational functions $k(t)$ in characteristic zero has finite index in the full automorphism group of the associated preimage tree. Moreover, we show that in most cases, the index is bounded independently of the polynomial.  
\end{abstract}
 
\section
\indent \indent\indent  When attempting to understand the Galois behavior of a family of polynomials over a number field, it is often informative to view the coefficients of these polynomials as indeterminates and study the corresponding Galois groups over the field of rational functions. This is helpful since Hilbert's irreducibility theorem implies that outside of a thin set, the Galois groups of the specialized polynomials are the same as those in the indeterminate case. Hence, by studying Galois groups over function fields, one has a handle on the generic situation over a number field. We apply this heuristic to the arboreal representations in quadratic dynamics.

Let us fix some notation. Let $k$ be a field of characteristic zero, let $K=K_0:=k(t)$ be the field of rational functions in one variable over $k$, and let $R:=k[t]$ be the polynomial ring. For $a\in K$, define $h(a):=deg(a)$ to be the logarithmic height of a rational function. We use the language of heights since we expect our statements, such as those in Theorem \ref{FIC}, to hold over number fields. 

To define the relevant dynamical Galois groups, let $\phi\in R[x]$ be a monic quadratic polynomial, and let $\phi^n$ be the $n^{th}$ iterate of $\phi$. We assume that $\phi^n$ is a separable polynomial for all $n\geq1$, so that the set $T_n(\phi)$ of roots of $\phi, \phi^2,\dots ,\phi^n$ together with $0$, carries a natural binary rooted tree structure: $\alpha,\beta\in T_n(\phi)$ share an edge if and only if $\phi(\alpha) =\beta$ or $\phi(\beta)=\alpha$. Furthermore, let $K_n:=K(T_n(\phi))$ and $G_n(\phi):=\Gal(K_n/K)$. Finally, we set
\begin{equation}{\label{Arboreal}}
T_\infty(\phi):=\bigcup _{n \geq 0} T_n(\phi)\;\;\text{and}\;\; G_\infty(\phi)=\varprojlim G_n(\phi).
\end{equation}  
Since $\phi$ is a polynomial with coefficients in $K$, it follows that $G_n(\phi)$ acts via graph automorphisms on $T_n(\phi)$. Hence, we have injections $G_n(\phi) \hookrightarrow \Aut(T_n(\phi))$ and $G_\infty(\phi) \hookrightarrow \Aut(T_\infty(\phi))$ called the \emph{arboreal representations} associated to $\phi$.   
 
A major problem in dynamical Galois theory, especially over global fields, is to understand the size of $G_\infty(\phi)$ in $\Aut(T_\infty(\phi))$. We prove that if $K$ is a field of rational functions in characteristic zero, then $G_\infty(\phi)$ is a finite index subgroup for many choices of $\phi$, including the family of quadratic polynomials $\phi_f(x)=x^2+f(t)$ for non-constant $f$; see Corollary \ref{uniform bound} below. Moreover, we show that in the generic case, the index bound does not depend on the polynomial.  

To do this, we relate the size of the relative extensions $K_n/K_{n-1}$ to the rational points on some curves defined by iteration. From there we use height bounds for rational points on Thue equations \cite{Schmidt} over function fields to bound $n$. The key assumption, which allows us to parametrize Galois behavior in terms of rational points, is that $\phi$ have the following stability condition.  
\begin{definition} We say that $\phi$ is \emph{stable} if all iterates of $\phi$ are irreducible polynomials. 
\end{definition} 
This is  a mild assumption which can be checked effectively; see Proposition \ref{stability} below. In addition to stability, we have the following definitions which also prove decisive when studying the Galois theory of iterates. 
\begin{definition}Let $\phi(x)=(x-\gamma)^2+c\in R[x]$. Then we set $h(\phi):=\max\{h(\gamma), h(c)\}$ and call $h(\phi)$ the height of $\phi$.    
\end{definition} 
\begin{definition} We say that $\phi=(x-\gamma)^2+c\in R[x]$ is \emph{post-critically finite} if the critical orbit $\mathcal{O}_\phi(\gamma):=\{\phi(\gamma),\phi^2(\gamma),\dots\}$ is a finite set and \emph{post-critically infinite} otherwise.   
\end{definition}  
\begin{definition}We say that $\phi$ is \emph{isotrivial} if there exists $g\in\Aut(\mathbb{A}^1)$ such that $\phi^g:=g\circ\phi\circ g^{-1}$ is defined over $k$ (i.e. $h(\phi^g)=0$).    
\end{definition} 
With the relevant background material and definitions in place, we prove a finite index theorem, with uniformity for a large class of examples, in quadratic dynamics over $K$.     
   
\begin{theorem}[\emph{Finite Index Theorem}]{\label{FIC}} Let $\gamma,c\in R$ and let $\phi(x)=(x-\gamma)^2+c\in R[x]$. If $\phi$ is stable, post-critically infinite and $h(\phi)>0$, then $G_\infty(\phi)$ is a finite index subgroup of $\Aut(T_\infty(\phi))$. In particular, the following bounds hold.  
\begin{enumerate} 
\item If $h(\gamma)\neq h(c)$, then 
\[\log_2\big|\Aut(T_\infty(\phi)):G_{\infty}(\phi)\big|\leq2^{16}-17=65519.\] Specifically, the extensions $K_n/K_{n-1}$ are maximal for all $n\geq17$.
\item If $h(\gamma)=h(c)$ and $h(\gamma-c)\neq0$, then \[\log_2\big\vert\Aut(T_\infty(\phi)):G_{\infty}(\phi)\big\vert<C\Bigg(\frac{h(\gamma)}{h(\gamma-c)}\Bigg).\]  Specifically, the extensions $K_n/K_{n-1}$ are maximal for all\, $n>2\cdot\log_2\Big(78\cdot\frac{h(\gamma)}{h(\gamma-c)}\Big)+9$.
\item If $h(\gamma-c)=0$ \,(i.e. $\phi$ is isotrivial), then 
\[\#\big\{n\,\big\vert \;K_{n}/K_{n-1}\;\text{is not maximal}\big\}\leq h(\gamma)-1.\] 
 \end{enumerate} In particular, for a large class of quadratic polynomials $\phi$ (those in part $1$), the index of $G_\infty(\phi)$ as a subgroup of $\Aut(T_\infty(\phi))$ is bounded independently of $\phi$. 
\end{theorem} 
After completing this manuscript, it was pointed out to me by Richard Pink how an improvement of the bound given in Theorem \ref{FIC} for isotrivial $\phi$ follows from \cite[Theorem 4.8.1]{Pink}; see Remark \ref{improvement} below.   
\begin{proof} For completeness, we include the following essential lemma of Stoll \cite[Lemma 1.6]{Stoll}, generalized to all monic quadratic polynomials over fields of characteristic not equal to $2$.
 \begin{lemma}{\label{Stoll}} Suppose that $\phi^{n-1}$ is irreducible. Then $|K_n:K_{n-1}|=2^{2^{n-1}}$ and $K_n/K_{n-1}$ is maximal, if and only if $\phi^n(\gamma)$ is not a square in $K_{n-1}$.   
 \begin{proof} We follow the proof sketch in \cite[\S 2.2]{R.Jones} which is generalization of \cite[Lemma 1.6]{Stoll}. Since $\phi$ is a quadratic polynomial, one sees that $K_n/K_{n-1}$ is a Kummer $2$-extension: let $X_{n-1}=\{\beta_1,\beta_2,\dots \beta_{2^{n-1}}\}$ be the roots of $\phi^{n-1}$. Then $K_n=K_{n-1}\big(\sqrt{\delta_1},\sqrt{\delta_2},\dots \sqrt{\delta_{2^{n-1}}}\big)$ for $\delta_i=\disc(\phi(x)-\beta_i)$. By Kummer theory, it follows that     
$D:=|K_n:K_{n-1}|$ is the order of the group generated by the $\delta_i$ in $K_{n-1}^*/(K_{n-1}^*)^2$. Here $(K_{n-1}^*)^2$ denotes the group of non-zero squares in $K_{n-1}$. Moreover, we have that 
\begin{equation}{\label{vector space}}
D=\frac{2^{2^{n-1}}}{\# V},\;\;\text{for}\;\; V:=\Big\{(e_1,e_2,\dots e_{2^{n-1}})\in\mathbb{F}_2^{2^{n-1}}:\; \prod_i\delta_i^{e_i}\in (K_{n-1}^*)^2\Big\}.
\end{equation} 
Note that $V$ is an $\mathbb{F}_2[G_{n-1}(\phi)]$-module. Since $G_{n-1}(\phi)$ is a $2$-group, it follows that $V^{G_{n-1}(\phi)}\neq0$ if and only if $V\neq0$. However, $\delta_i=-4\cdot(c-\beta_i)$ and $G_{n-1}(\phi)$ acts transitively on $X_{n-1}$. Hence $G_{n-1}(\phi)$ acts transitively on the $\delta_i$'s also. Therefore, $K_{n}/K_{n-1}$ is not maximal (i.e. $V\neq0$) if and only if $(1,1,\dots 1)\in V$. Moreover, we have that $(1,1,\dots,1)\in V$ is precisely the statement that
\[\prod_i \delta_i=(-4)^{2^{n-1}}\cdot\prod_i(c-\beta_i)=(-4)^{2^{n-1}}\cdot\phi^{n-1}(c)=(-4)^{2^{n-1}}\cdot \phi^n(\gamma) \in (K_{n-1}^*)^2,\] by definition of $V$ in (\ref{vector space}). This completes the proof of Lemma \ref{Stoll}.             
 \end{proof}     
 \end{lemma}
 We return to the proof of Theorem \ref{FIC}. Suppose that $K_n/K_{n-1}$ is not maximal. We bound $n$ according to those cases specified above. Since $\phi$ is stable, Lemma \ref{Stoll} implies that $\phi^n(\gamma)$ is a square in $K_{n-1}$. If we write $\phi^n(\gamma)=d_n\cdot y_n^2$ for some $y_n,d_n\in R$ such that $d_n$ is a unit or a square-free polynomial, then the primes dividing $d_n$ must ramify in $K_{n-1}$.
 
On the other hand, by \cite[Corollary 2, p.159]{Number theory} and \cite[Proposition 7.9]{Rosen}, we see that the primes which ramify in $K_{n-1}$ must divide the discriminant of $\phi^{n-1}$. Let $\Delta_m$ be the discriminant of $\phi^m$. Then we have the following formula, given in \cite[Lemma 2.6]{Jones2}: 
 \begin{equation}{\label{discriminant}}\Delta_m=\pm\Delta_{m-1}^2\cdot 2^{2^n}\cdot \phi^m(\gamma).
 \end{equation} 
 In particular, if $d_n$ is not a unit, then $d_n=\prod p_i$ for some primes $p_i\in R$ such that $p_i\big\vert\phi^{m_i}(\gamma)$ and $1\leq m_i\leq n-1$. On the other hand, since $p_i\big\vert\phi^{m_i}(\gamma)$ and $p_i\big\vert\phi^{n}(\gamma)$, it follows that $p_i\big\vert\phi^{n-m_i}(0)$. To see this, note that \[\phi^{n-m_i}(0)\equiv\phi^{n-m_i}(\phi^{m_i}(\gamma))\equiv\phi^n(\gamma)\equiv0\bmod{p_i}.\] In particular, we have the refinement,   
 \begin{equation}{\label{refinement}}  d_n=\prod p_i,\;\;\text{where}\;\; p_i\big\vert\phi^{t_i}(\gamma)\;\text{or}\; p_i\big\vert\phi^{t_i}(0)\;\; \text{for some}\; 1\leq t_i\leq\Big\lfloor \frac{n}{2}\Big\rfloor.   
\end{equation} 
Moreover, it follows that
\begin{equation}{\label{point}} \big(\,\phi^{n-1}(\gamma)\,, \;y_n\cdot d_n\cdot(\phi^{n-2}(\gamma)-\gamma)\big)\in E_\phi^{(d_n)}(k[t])
\end{equation} 
is a point on the elliptic curve $E_\phi^{(d_n)}$ (each $n$ yielding a quadratic twist of a fixed curve $E_\phi$) defined by the equation  
\begin{equation}{\label{curve}}
E_\phi^{(d_n)}:\; Y^2=d_n\cdot(X-c)\cdot\big((X-\gamma)^2+c\,\big).
\end{equation}
Indeed, $E_\phi^{(d_n)}$ is non-singular since $\phi$ is irreducible. 

Our goal is to use height bounds on rational points on curves defined over function fields to obtain a bound on $n$. To do this, we need some elementary estimates for the heights of points in $\mathcal{O}_\phi(\gamma)$ and $\mathcal{O}_\phi(0)$, stated in the following lemma.
\begin{remark} The arithmetic of curves defined by iteration, such as those in (\ref{curve}), are studied in some detail over number fields and finite fields in \cite{Me} and \cite{Me1}. 
\end{remark}     
\begin{lemma}{\label{Heights}} Let $\phi(x)=(x-\gamma)^2+c$ and define $h(\phi):=\max\{h(\gamma), h(c)\}$. Then we have the following height estimates for points in $\mathcal{O}_\phi(\gamma)$ and $\mathcal{O}_\phi(0)$:
 \begin{enumerate} 
 \item If $h(\gamma)\neq h(c)$, then the height bounds below hold: 
 \begin{enumerate} \item $h(\phi^m(\gamma))=2^{m-1}\cdot h(\phi)$ for all $m\geq2$, and $h(\phi(\gamma))\leq h(\phi)$.   
 \item $h(\phi^m(0))\leq2^{m}\cdot h(\phi)$ for all $m\geq1$. 
 \end{enumerate} 
 \item If $h(\gamma)=h(c)$ and $h(\gamma-c)>0$, then let $\rho_\phi:=\log_2\Big(\frac{h(\gamma)}{h(\gamma-c)}\Big)+1$. In this case, we have the following bounds. 
 \begin{enumerate} 
 \item $h(\phi^m(\gamma))\leq h(\gamma)=2^{\rho_\phi-1}\cdot h(\gamma-c)$ for all $m\leq\rho_\phi$. 
 \item $h(\phi^m(\gamma))=2^{m-1}\cdot h(\gamma-c)$ for all $m>\rho_\phi$.   
 \item $h(\phi^m(0))=2^m\cdot h(\gamma)=2^{m+\rho_\phi-1}\cdot h(\gamma-c)$ for all $m\geq1$. 
 \end{enumerate}    
 \end{enumerate} 
 \begin{proof} Suppose that $h(\gamma)\neq h(c)$. Clearly, $\phi(\gamma)=c\leq h(\phi)$ by definition of $h(\phi)$. On the other hand, we see that \[h(\phi^2(\gamma))=h((c-\gamma)^2+c)=\max\{2h(\gamma),2h(c)\}=2h(\phi),\] since $h(\gamma)\neq h(c)$. Similarly, $\phi^3(\gamma)=(\phi^2(\gamma)-\gamma)^2+c$. But we have shown that $h(\phi^2(\gamma))>h(\phi)$, and hence $h(\phi^3(\gamma))=2h(\phi^2\gamma)=4h(\phi)$. One proceeds by induction in this way to prove that $h(\phi^m(\gamma))=2h(\phi^{m-1}(\gamma))=2^{m-1}\cdot h(\phi)$. 
 
 Similarly, we deduce height bounds for points in $\mathcal{O}_\phi(0)$. Clearly, $h(\phi(0))=h(\gamma^2+c)\leq\max\{2h(\gamma),h(c)\}\leq2h(\phi)$. Hence, by induction we have that \[h(\phi^m(0))\leq\max\{2\cdot h\big(\phi^{m-1}(0)-\gamma\big),h(c)\}\leq\max\{2\cdot h(\phi^{m-1}(0)),2\cdot h(\gamma),h(c)\}\leq2^m\cdot h(\phi),\] which finishes the proof of part $1$.  

Now suppose that $h(\gamma)=h(c)$ and that $h(\gamma-c)>0$. Let $b:=\gamma-c$. One sees that $\phi^m(\gamma)=(((b^2+b)^2+b)^2\dots+b)^2+c$. Hence, $h(\phi^m(\gamma))\leq\max\{2^{m-1}\cdot h(b),h(c)\}$ and equal to this maximum whenever $2^{m-1}\cdot h(b)\neq h(c)$. However, since $h(b)>0$, this is precisely when $m\neq\rho_\phi$, as $2^{\rho_\phi-1}\cdot h(b)=h(c)$. This establishes cases $(a)$ and $(b)$ of part $2$. 

Finally, we prove the height bounds for points in $\mathcal{O}_\phi(0)$ stated in part $2$. Clearly, $\phi(0)=\gamma^2+c$. Since, $h(\gamma)=h(c)>0$, we see that $h(\phi(0))=2\cdot h(\gamma)$. Similarly, $\phi^2(\gamma)=(\gamma^2+c-\gamma)^2+c$ and $h(\phi^2(0))=4\cdot h(\gamma)$. One continues in this way to establish part $(c)$ by induction. This completes the proof of Lemma \ref{Heights}.           
 \end{proof}  
\end{lemma}
\begin{remark} Note that Lemma \ref{Heights} implies that $\mathcal{O}_\phi(\gamma)$ is bounded if and only if $h(\gamma-c)=0$, or equivalently if and only if $\phi$ is isotrivial.  
\end{remark} 
We return to the proof of Theorem \ref{FIC} and use \cite[Theorem 6]{Mason} applied to the the curve (\ref{curve}) and the point (\ref{point}) to obtain the bound
\begin{equation}{\label{Mason}}
h_{K_1}(\phi^{n-1}(\gamma))\leq26\cdot h_{K_1}\big(E_\phi^{(d_n)}\big)+8\cdot g_{K_1}+4\cdot(r_{K_1}-1).
\end{equation} 
Here $h_{K_1}$ denotes the logrithmic height associated to the function field $K_1$ and $h_{K_1}\big(E_\phi^{(d_n)}\big)$ is the maximum height (relative to $K_1$) of the coefficients of $d_n\cdot(X-c)\cdot\phi(X)$; see the introduction of \cite{Schmidt} for the relevant background. Moreover, $g_{K_1}$ and $r_{K_1}$ denote the genus and number of infinite places of $K_1$ respectively.  

In particular, \cite[Lemma H]{Schmidt} and (\ref{refinement}) together imply that
\begin{small}\begin{equation}{\label{Thue}}h(\phi^{n-1}(\gamma))\leq26\cdot\Big( h\big(\phi(\gamma)\big)+\dots +h\big(\phi^{\lfloor\frac{n}{2}\rfloor}(\gamma)\big)+h\big(\phi(0)\big)+\dots +h\big(\phi^{\lfloor\frac{n}{2}\rfloor}(0)\big)+3h(\phi)\Big)+32\cdot h(\phi)+4.
\end{equation}\end{small} 
We are ready to use the height bounds in Lemma \ref{Heights}.  

\textbf{Part (1).} If $h(\gamma)\neq h(c)$, then part $1$ of Lemma \ref{Heights} implies that
\[2^{n-2}\cdot h(\phi)\leq 26\cdot h(\phi)\cdot\big(1+2+\cdots+2^{\lfloor\frac{n}{2}\rfloor-1}+2+\cdots+2^{\lfloor\frac{n}{2}\rfloor}+3\big)+32\cdot h(\phi)+4.\] 
Note that since $h(\gamma)\neq h(c)$, we have that $h(\phi)\neq0$. In particular, under the hypotheses of Theorem $\ref{FIC}$, if $K_n/K_{n-1}$ is not maximal and $h(\gamma)\neq h(c)$, then 
\begin{equation}\frac{2^{n-2}}{2^{\lfloor\frac{n}{2}\rfloor+1}+2^{\lfloor\frac{n}{2}\rfloor} +1}\leq36. 
\end{equation} 
Hence $n\leq16$. As for the index bound, it suffices to show that $\log_2|\Aut(T_n(\phi)):G_n(\phi)|\leq2^{16}-17$ for all $n\geq17$. To do this, note that \[ |K_{17}:K|=|K_{17}:K_{16}|\cdot|K_{16}:K|=2^{2^{16}}\cdot|K_{16}:K|\geq2^{2^{16}}\cdot2^{16}=2^{2^{16}+16},\] since $K_{17}/K_{16}$ is maximal and $\phi^{16}$ is an irreducible polynomial of degree $2^{16}$. However, since $\Aut(T_n)$ is the $n$-fold wreath product of $\mathbb{Z}/2\mathbb{Z}$, it follows that $|\Aut(T_n)|=2^{2^n-1}$. Hence, 
\[|\Aut(T_{17}(\phi)):G_{17}(\phi)|=\frac{2^{2^{17}-1}}{|K_{17}:K|}\leq2^{(2^{17}-1)-(2^{16}+16)}=2^{2^{16}-17}\] as claimed. Similarly, we see that for all $n\geq17$,  
\[\log_2|\Aut(T_{n}(\phi)):G_{n}(\phi)|=\log_2\frac{2^{2^{n}-1}}{|K_{16}:K|\cdot\prod_{j=16}^{n-1}|K_{j+1}:K_j|}\leq2^n-1-(2^{n-1}+2^{n-2}+\dots +2^{16}+16),\] since all of the subextensions $K_{j+1}/K_j$ are maximal. Furthermore, \[2^n-1-(2^{n-1}+2^{n-2}+\dots +2^{16}+16)=2^n-1-(2^n-1-(2^{16}-1)+16)=2^{16}-17,\] and we get the index bound in part $1$ of Theorem \ref{FIC}. 

\textbf{Part (2).} Similarly, suppose that $h(\gamma)=h(c)$ and that $h(\gamma-c)\neq0$. Let $\rho_\phi:=\log_2\Big(\frac{h(\gamma)}{h(\gamma-c)}\Big)+1$ as in part $2$ of Lemma \ref{Heights}. If $K_n/K_{n-1}$ is not maximal and $n>2\cdot\log_2\Big(78\cdot\frac{h(\gamma)}{h(\gamma-c)}\Big)+9$, then Lemma \ref{Heights} and (\ref{Thue}) imply that \[2^{n-2}\cdot h(\gamma-c)\leq 26\cdot h(\gamma-c)\cdot\big(2\cdot(2^{\rho_\phi}+2^{\rho_\phi+1}+\dots+2^{\lfloor\frac{n}{2}\rfloor+\rho_\phi-1})+3\cdot2^{\rho_\phi-1}\big)+32\cdot2^{\rho_\phi-1}\cdot h(\gamma-c)+4,\] since in particular $n-1>\rho_\phi$. From here, we obtain the bound \[2^{n-2}\leq78\cdot2^{\lfloor\frac{n}{2}\rfloor+\lambda+1}.\] However, this forces $n\leq2\cdot\log_2\Big(78\cdot\frac{h(\gamma)}{h(\gamma-c)}\Big)+9$, which contradicts our assumption on $n$.

\textbf{Part (3).} Finally, assume that $h(\gamma-c)=0$. It follows that $\phi^n(\gamma)=c+c_n$ for some constant functions $c_n\in k$. In particular, $h(\phi^n(\gamma))=h(\gamma)=h(c)>0$ for all $n\geq1$. Since $\mathcal{O}_\phi(\gamma)$ is infinite, $\phi^m(\gamma)\neq \phi^n(\gamma)$ for all $n\neq m$. Equivalently, $c_n\neq c_m$ for all $n\neq m$. Therefore, 
\begin{equation}{\label{gcd}} \gcd(\phi^n(\gamma),\phi^m(\gamma))=1,\;\;\text{for all}\;n\neq m.
\end{equation} 
Now, suppose that $K_n/K_{n-1}$ is not maximal and write $\phi^n(\gamma)=d_n\cdot y_n^2$ as above. Then (\ref{refinement}) and (\ref{gcd}) together imply that $d_n$ must be constant. In particular, since $h(\phi^n)>0$, we see that  $h(y_n)>0$. Hence the discriminant $\disc_{\,t}(\phi^n(\gamma))=0$.

Let $p(s):=\disc_{\,t}(c(t)+s)\in k[s]$. One checks that $\deg(p)=h(c)-1=h(\gamma)-1$. However, we have shown that if $K_n/K_{n-1}$ is not maximal, then $p(c_n)=0$. Hence,
\begin{equation}{\label{count isotrivial}}
\#\big\{n\,\big\vert \;K_{n}/K_{n-1}\;\text{is not maximal}\big\}\leq\deg(p)=h(\gamma)-1.
\end{equation} 
This completes the proof of Theorem \ref{FIC}.                             
 \end{proof} 
 \begin{remark}{\label{improvement}} Richard Pink has pointed out to me how to improve the bound in part $3$ of Theorem \ref{FIC} using the following argument: consistent with the setup in \cite{Pink}, let $f(x)=x^2+a\in k[x]$ be a post-critically infinite polynomial. For any transcendental $s\in k(t)$, let $G_{k(t),s}$ denote the inverse limit of the Galois groups over $k(t)$ of $f^n(x)-s$. It was shown in \cite[Theorem 4.8.1]{Pink} that $G_{k(s),s}\cong\Aut(T_\infty)$. Now set $d:=|k(t):k(s)|$ and let $e=\ord_2(d)$ be the maximal exponent of $2$ in $d$. Then $G_{k(t),s}$ is a subgroup of $G_{k(s),s}\cong\Aut(T_\infty)$ of index at most $2^{e}$. In our situation, apply the transformation $x\rightarrow x+\gamma$ to $\phi$ and take $s=\gamma$ so that $G_\infty(\phi) = G_{k(t),s}$.
 \end{remark} 
We note how Theorem \ref{FIC} is also useful when studying the Galois theory of certain non-stable polynomials in $k(t)[x]$ (not necessarily polynomial coefficients). To see this, we fix some notation. Given $\phi(x)=(x-\gamma)^2+c\in R[x]$ and any rational function $f\in k(t)$, define
\[\phi_f(x):=\big(x-\gamma(f)\big)^2+c(f).\] With this definition in place, Tom Tucker has pointed out the following corollary of Theorem \ref{FIC}.
\begin{corollary}{\label{Cor1}} Let $\phi(x)=(x-\gamma)^2+c\in R[x]$ satisfy the conditions of Theorem \ref{FIC} and let $f\in k(t)$ be non-constant. Then $|\Aut(T_\infty): G_\infty(\phi_f)|$ is finite. In particular, the number of irreducible factors of $\phi_f^n$ is bounded independently of $n$.   
\end{corollary}
\begin{proof} We apply Theorem \ref{FIC} to the rational function field $K_0:=k(f)\cong k(t)$, obtaining index bounds for the arboreal representation over the ground field $k(f)$. We then use the fact that the degree of the extension $k(t)/k(f)$ is equal to $h(f)$. In particular, one sees that \[|\Aut(T_\infty(\phi_f)):G_\infty(\phi_f)|\leq |\Aut(T_\infty): G_\infty(\phi)|\cdot |k(t):k(f)|,\] with $|\Aut(T_\infty): G_\infty(\phi)|$ bounded as in Theorem \ref{FIC}.   
\end{proof}    
As an example of how to apply Corollary \ref{Cor1}, let $\phi(x)=(x-t^3+1)^2-t$. One can use Proposition \ref{stability} below to prove that $\phi$ is stable. In particular, Corollary \ref{Cor1} implies that $G_\infty(\phi_{1/t^2})$ has finite index in $\Aut(T_\infty)$, even though \[\phi_{1/t^2}\big(x\big)=\Big(x+\frac{t^6-t^5-1}{t^6}\Big)\cdot\Big(x+\frac{t^6+t^5-1}{t^6}\Big)\] is reducible. This illustrates how one can circumvent stability. 

However, in the special family $\phi_f(x)=x^2+f(t)$, one can do better than the index bound given in the proof of Corollary \ref{Cor1}, which will grow with the degree of $f$. 
 \begin{corollary}{\label{uniform bound}}Let $f\in R$ be a non-constant polynomial such that $-f$ is not a square, and let $\phi_f(x)=x^2+f$. Then the uniform index bound in part $1$ of Theorem \ref{FIC} holds for $G_\infty(\phi_f)$. 
 \begin{proof}Since the critical point $\gamma=0$ and $f$ is non-constant (hence $h(f)\neq h(\gamma)$ as in part $1$ of the theorem), it suffices to check that $\phi_f$ is stable to apply Theorem \ref{FIC}. However, by \cite[Proposition 4.2]{Jones2}, it suffices to show that the adjusted critical orbit $\{-\phi_f(0),\phi_f^2(0),\phi_f^3(0),\dots\}$ contains no perfect squares. By assumption, $-\phi_f(0)=-f$ is not a square. On the other hand, suppose that \[g^2=\phi^n_f(0)=((((f^2+f)^2+f)^2+f\dots)^2+f\] for some polynomial $g\in R$ and some $n\geq2$. Hence, $g^2=f^2\cdot h^2+f=f\cdot(f\cdot h^2+1)$ for some $h\in R$. Since, $f$ and $f\cdot h^2+1$ are coprime and $R$ is a UFD, it follows that $f=s^2$ and $f\cdot h^2+1=t^2$ for some $s,t\in R$. In particular, $1=(t-h\cdot s)(t+h\cdot s)$ and both $t-h\cdot s$ and $t+h\cdot s$ must be constant. Hence, $2h\cdot s=(t+h\cdot f)-(t-h\cdot f)$ must be constant. This is impossible, since $f$ (and hence $s$) is assumed to be non-constant.           
 \end{proof}
 \end{corollary} 
 \begin{remark}Consistent with the notation in \cite[Theorem 1.1]{Me}, we have proven that $S^{(n)}(R)=\varnothing$ for all $n\geq17$. This answers a question, posed at the end of \cite{Me} for number fields, in the rational function field setting.   
\end{remark} 
We now use height bounds for rational points on curves to show how to effectively determine if a non-isotrivial polynomial is stable, crucial if one wishes to bound the index of the arboreal representation. 
\begin{proposition}{\label{stability}} Let $\gamma,c\in R$ and let $\phi(x)=(x-\gamma)^2+c\in R[x]$ be non-isotrivial. If $c\cdot \phi(c)\neq0$, then $\phi$ is stable whenever one of the following conditions is satisfied.
\begin{enumerate} 
\item $h(\gamma)\neq h(c)$ and the set $\big\{-\phi(\gamma),\phi^2(\gamma),\phi^3(\gamma),\dots \phi^{8}(\gamma) \big\}$ does not contain any squares. 
\item $h(\gamma)=h(c)$ and the set $\big\{-\phi(\gamma),\phi^2(\gamma),\phi^3(\gamma),\dots \phi^{[s_\phi]}(\gamma) \big\}$ does not contain any squares, where $s_\phi:=\log_2\Big(110\cdot\frac{h(\gamma)}{h(\gamma-c)}\Big)+3$ and $[ \cdot]$ denotes the nearest integer function.  
\end{enumerate} 
\end{proposition}
\begin{proof} By \cite[Proposition 4.2]{Jones2}, it suffices to show that the set $\{-\phi_f(0),\phi_f^2(0),\phi_f^3(0),\dots\}$ contains no squares, to ensure that $\phi$ is stable. It is our goal to use height bound arguments, similar to those in Theorem \ref{FIC}, to bound to the largest iterate one must check.

If $\phi^n(\gamma)=y_n^2$ for some $n\geq2$, then 
\[ \big(\phi^{n-1}(\gamma),y_n\cdot (\phi^{n-2}(\gamma)-\gamma) \big)\;\;\text{is a point on}\;\; E_\phi: Y^2=(X-c)\cdot\phi(X).\] 
Moreover, $E_\phi$ is nonsingular since $c\cdot \phi(c)\neq0$. Let $K_1/k(t)$ be the splitting field of $\phi$ over $k(t)$. Then \cite[Theorem 6]{Mason} implies that  
\[h_{K_1}(\phi^{n-1})\leq26\cdot h_{K_1}(E_\phi)+8g_{K_1}+4(r_{K_1}-1).\] 
In particular, if $h(\gamma)\neq h(c)$, then \cite[Lemma H]{Schmidt} and Lemma \ref{Heights} imply that 
\[ 2^{n-2}\cdot h(\phi)\leq110\cdot h(\phi)+4.\] 
Hence, $n\leq8$ as claimed. Similarly, if $h(\gamma)=h(c)$ and $n>s_\phi$ (in particular $n-1>\rho_\phi$), then Lemma \ref{Heights} implies that 
\[2^{n-2}\cdot h(\gamma-c)\leq78\cdot2^{\rho_\phi-1}\cdot h(\gamma-c)+32\cdot 2^{\rho_\phi-1}\cdot h(\gamma-c)+4.\]
Hence, $2^{n-2}\leq110\cdot 2^{\rho_\phi}$ and $n\leq s_\phi$. This contradicts our assumption on $n$. In particular, if $\phi^n(\gamma)$ is a square, then $n$ is bounded as claimed.  
\end{proof}   
We take the time now to note that the index bounds in Theorem \ref{FIC} are not sharp; see Proposition \ref{examples} below for isotrivial examples. As for non-isotrivial examples, let $k=\mathbb{Q}$ and $\phi=x^2+t$. One can show directly with Magma \cite{Magma} that $\phi$ is stable using Proposition \ref{stability}. Moreover, it can also be checked with Magma that $\phi^n(0)$ contains square-free primitive prime divisors for all $n\leq17$. That is to say, for all $2\leq n\leq 16$, not all primes appearing in the factorization of $d_n$ (as in \ref{refinement}) divide lower iterates. In particular, $G_\infty(\phi)\cong\Aut(T_\infty)$ by the proof of Theorem \ref{FIC} part $1$. This greatly improves the index bound given in the theorem.   

As for a non-trivial index example, let $\phi=x^2+t$ and $f=-t^2-1$. Then by Corollary \ref{Cor1} and the previous paragraph, we see that $|\Aut(T_\infty(\phi_{f})):G_\infty(\phi_{f})|\leq2$. However, one checks directly that the Galois group of second iterate has index $|\Aut(T_2(\phi_{f})):G_2(\phi_{f})|=2$. Hence, the full index of $G_\infty(\phi_{f})$ in $\Aut(T_\infty(\phi_{f}))$ is $2$, significantly less than the index bound given in Theorem \ref{FIC} for the stable polynomial $x^2+(-t^2-1)$. 

This leads to the problem of determining the smallest possible upper bound, depending on $k$, for the index of $G_\infty(\phi)$ inside $\Aut(T_\infty(\phi))$ as we range over all stable polynomials satisfying $h(\gamma)\neq h(c)$. It is likely much smaller than that given in Theorem \ref{FIC}.   
       
Similarly, one can drastically improve the bounds for the index of the arboreal representation in certain families of isotrivial polynomials using the proof of Theorem \ref{FIC} part $3$, instead of the statement. Specifically, consider the family
\[\phi(x):=\big(x-t^d\big)^2+t^d+m, \;\;\;\;\text{for}\;\;d\geq1\;\; \text{and}\;\; m\in \mathbb{Q}\setminus\{0,-1,-2\}.  \]     
It is easy to see that \[\mathcal{O}_{\phi}(t^d)=\big\{t^d+m,\, t^d+(m^2+m),\, t^d+(((m^2+m)^2+m),\dots\big\},\]
 is post-critically infinite. Moreover, $h(t^d)=d=h(t^d+m)$ and $h(t^d-(t^d+m))=0$. Hence, according to Theorem \ref{FIC} part $3$, the bound on the index of $G_\infty(\phi)$ in $\Aut(T_\infty(\phi))$ weakens as $h(\phi)$ becomes large (or with Pink's refinement given in Remark \ref{improvement}, as the $2$-part of $d$ grows). However, as a generalization of \cite[Theorem 4.1]{B-J}, we use the ideas in the proof of Theorem \ref{FIC} to strengthen the index bounds in this particular family.    
 \begin{proposition}{\label{examples}} For $d\geq1$, let $\phi(x):=\big(x-t^d\big)^2+(t^d+m)\,\in k[t]$ and let $m$ be a constant function. If $\mathcal{O}_\phi(t^d)$ is infinite, then  
 \[\Aut(T_\infty(\phi))\cong G_\infty(\phi).\] In particular, if $k=\mathbb{Q}$ and $m\notin\{0,-1,-2\}$, then $G_\infty(\phi)\cong\Aut(T_\infty(\phi))$.   
\begin{proof} Let $f(s)=s^2+m$, so that $\phi^n(\gamma)=t^d+f^n(0)$. Since $\mathcal{O}_\phi(\gamma)$ is infinite, it follows that $f^n(0)\neq f^r(0)$ for all $n\neq r$. In particular, $f^n(0)\neq0$ for all $n\geq1$. Hence the discriminant of both $-\phi(\gamma)$ and $\phi^n(\gamma)$ are non-zero for all $n\geq1$. It follows from \cite[Proposition 4.2]{Jones2}, that $\phi$ is a stable polynomial. 

On the other hand, since $\phi^n(t^d)-\phi^r(t^d)\in k^*$, it follows that $gcd\big(\phi^n(t^d),\phi^r(t^d)\big)=1$ for all $n\neq r$. Hence, if the extension $K_n/K_{n-1}$ is not maximal, then (\ref{refinement}) implies that $\phi^n(\gamma)=d_n\cdot y_n^2$ for some unit $d_n$. However, we have shown that $\phi^m(\gamma)$ is square-free (non-zero discriminant), hence $y_n$ must also be a constant. This implies that $\phi^m(\gamma)$ is a constant, contradicting our assumption that $d\geq1$. It follows that $\Aut(T_\infty(\phi))\cong G_\infty(\phi)$. 

For the case when $k=\mathbb{Q}$, suppose that $\phi$ is post-critically finite. Equivalently, $f(s)=s^2+m$ must be post-critically finite. However, if $f$ is post-critically finite over $\mathbb{Q}$, then $m$ belongs to the Mandelbrot set $\mathcal{M}$ over the complex numbers; see \cite[\S4.24]{Silv-Dyn}. In particular, \cite[Proposition 4.19]{Silv-Dyn} implies that $|m|\leq2$, where $|\cdot|$ denotes the complex absolute value. Hence the absolute logarithmic height of $m$ is at most $\log(2)$. One checks that this implies that $m\in\{0,-1,-2\}$.                
\end{proof} 
\end{proposition}
One would like to weaken the stability hypothesis of Theorem \ref{FIC} to that of \emph{eventual stability}, meaning that the number of irreducible factors of $\phi^n$ is bounded independently of $n$. This condition is known to hold in many settings; see \cite[Theorem 5]{Eventually} and \cite[Corollary 3]{Ingram}. In particular, it may be the case that if $\phi$ is post-critically infinite and $0$ does not lie in a periodic orbit of $\phi$ (weaker than stability), then the image of $G_\infty(\phi)$ inside $\Aut(T_\infty(\phi))$ is a finite index subgroup.        

Similarly, it is tempting to think that the same results hold over all function fields in characteristic zero. This is to be expected, especially since the main ingredients, Stoll's lemma and explicit height bounds, go through without any problems (the bounds just depend on the genus). 

However, the difficulty arises when attempting to formulate the non-maximality of $K_n/K_{n-1}$ in terms of the rational points on curves. In the proof of Theorem \ref{FIC}, we wrote $\phi^n(\gamma)=d_n\cdot y_n^2$ for some elements  $d_n,\,y_n\in K$. This is not possible in general. For instance, if $K=k(C)$ and $C$ has genus $1$, then this relationship says that the square-free part or the divisor of $\phi^n(\gamma)$ is the divisor of a function, hence has degree zero and is comprised of dependent points on the elliptic curve. But we need this to hold for all $n$ and all stable $\phi$ to get uniform bounds, which won't be the case. 

However, for a particular $C$, it would be interesting to explicitly describe the set of all quadratic  polynomials $\phi$ for which a uniform bound on $|\Aut(T_\infty(\phi)):G_\infty(\infty)|$ can be obtained.    
\\
   
 \textbf{Acknowledgements:} It is a pleasure to thank Richard Pink for his comments on this manuscript and for pointing out how to improve the index bound for isotrivial polynomials given in part $3$ of Theorem \ref{FIC}. I also thank Joe Silverman, Tom Tucker, and Rafe Jones for the many useful discussions related to this work. Finally, I would like to thank the Arizona Winter School for motivating various number theoretic questions over function fields.

\end{document}